\documentclass[12pt]{amsart}

\calclayout

\usepackage[alphabetic]{amsrefs}

\usepackage{amssymb}

\usepackage{mathtools}

\usepackage{tikz-cd}

\usepackage{slashed}

\usepackage[all,cmtip]{xy}

\usepackage{hyperref}

\newcommand{\R}{\mathbb{R}}

\newcommand{\Z}{\mathbb{Z}}
\newcommand{\N}{\mathbb{N}}

\DeclareMathOperator{\Ima}{Im}
\DeclareMathOperator{\Coeff}{Coeff}

\DeclareMathOperator{\ind}{ind}

\numberwithin{equation}{section}

\theoremstyle{plain}
\newtheorem{thm}[equation]{Theorem}
\newtheorem{cor}[equation]{Corollary}
\newtheorem{lem}[equation]{Lemma}
\newtheorem{prop}[equation]{Proposition}

\theoremstyle{definition}
\newtheorem{defn}[equation]{Definition}

\theoremstyle{remark}
\newtheorem{rem}[equation]{Remark}

\begin{document}

\title{Augmentation categories in higher dimensions}

\author{Hanming Liu}

\thanks{This research was partly supported by NSF Grant DMS-2204214.}

\address{Department of Mathematics, University of Oregon}


\email{hanming@uoregon.edu}


\date{\today}

\begin{abstract}
	For an exact symplectic manifold \(M\) and a Legendrian submanifold \(\Lambda\) of the contactization \(M\times\R\), we construct the augmentation category (over a field of characteristic 2), a unital \(A_\infty\)-category whose objects are augmentations of the Chekanov-Eliashberg differential graded algebra. This extends the construction of the augmentation category by Ng-Rutherford-Shende-Sivek-Zaslow to contact manifolds of dimension greater than 3.
\end{abstract}

\maketitle

\tableofcontents

\section{Introduction}

In \cite{BC}, Bourgeois-Chantraine constructed, for every compact Legendrian submanifold in a \(1\)-jet space, a category \(\mathcal{A}ug_-\), which is a non-unital \(A_\infty\)-category whose objects are augmentations of the Chekanov-Eliashberg DGA, using algebraic methods. In \cite{NRSSZ}, Ng-Rutherford-Shende-Sivek-Zaslow constructed, for every Legendrian link \(\Lambda\in\R^3\), a category \(\mathcal{A}ug_+\), which is a unital \(A_\infty\)-category whose objects are augmentations of the Chekanov-Eliashberg DGA of \(\Lambda\), using parallel copies of \(\Lambda\). We extend the construction of \(\mathcal{A}ug_+\) (over a field of characteristic 2) to any compact Legendrian submanifold \(\Lambda\) of the contactization \(V=M\times\R\), where \(M\) is an exact symplectic manifold with finite geometry at infinity (the important case to keep in mind is when \(V=T^*X\times\R\) is a \(1\)-jet space). In the case of \(M=T^{*}\R\), our construction agrees with the construction in \cite{NRSSZ}. Roughly, the main theorem of \cite{NRSSZ} says that in for a Legendrian link \(\Lambda\subset \R^3\), \(\mathcal{A}ug_+(\Lambda)\) is equivalent to a certain category of sheaves with microlocal rank 1.

Very roughly, \(\mathcal{A}ug_+(V,\Lambda)\) is defined as follows: objects are augmentations of the Chekanov-Eliashberg DGA of \(\Lambda\), morphisms are cochain complexes, which as graded vector spaces, are each spanned by Reeb chords from \(\Lambda\) to a positive push off of \(\Lambda\). The differential \(m_1\), and also higher \(A_\infty\)-operations \(m_k\), are defined using the Chekanov-Eliashberg DGA of \(k+1\) parallel copies of \(\Lambda\) equipped with augmentations.

In \cite{NRSSZ}, in order for the operations to satisfy the \(A_\infty\)-relations, they used specific perturbation schemes for the parallel copies, and obtained so-called ``consistent sequences of DGAs'' in the case of Legendrian links in \(\R^3\). In higher dimensions, it is more difficult to find appropriate perturbation schemes that would give rise to consistent sequences of DGAs. Moreover, even if one finds an appropriate perturbation scheme, it is not clear how to prove invariance of the category defined.

In this paper, we use an analogue of an idea of Abouzaid-Seidel on constructing wrapped Fukaya categories with an exposition written down in \cite{GPS1}. We use arbitrary perturbations of parallel copies, with no consistency requirement. We build a much larger category, which has a countably infinite poset of objects for every augmentation, and only allow the morphisms to ``go up'' in the poset. We then localise at the ``identity'' morphisms. It is then easier to prove invariance. Our invariance proof also serves as a new proof of invariance of the construction in \cite{NRSSZ}, see Remark \ref{invariance of aug+ from NRSSZ} for detail. We thank a referee for pointing out that a similar strategy was sketched in \cite{Cha}*{Section 3.3}.

In this paper, we work over a field of characteristic 2. In order to carry out our construction over an arbitrary field, one needs to work with Legendrians equipped with spin structures, and build a coherent orientation scheme for moduli spaces of quantum flow trees, and prove analogues of Lemma \ref{cocycle} and Proposition \ref{locinvlem}. The rest of this paper does not use the characteristic 2 assumption.

The organisation of the paper is as follows. In section \ref{background}, we review some background on Legendrian contact homology, how to count holomorphic curves in parallel copies using quantum flow trees, and some algebraic notions. In section \ref{construction and invariance}, we construct the augmentation category, and prove our main result (see Theorem \ref{invariance of aug+} for the more precise statement):

\begin{thm}
	Up to \(A_\infty\)-equivalence, \(\mathcal{A}ug_+(V,\Lambda)\) is independent of choices of perturbation and almost complex structure. Furthermore, it is invariant under Legendrian isotopy.
\end{thm}

In section \ref{equivalence with consistent sequences}, we prove that if we have parallel copies that give a consistent sequence of DGAs, then our construction coincides with a straightforward generalisation of the construction in \cite{NRSSZ}.

\section*{Acknowledgments}

I would like to thank Lenhard Ng for suggesting this problem, and for many helpful discussions. I would like to thank my advisor Robert Lipshitz for his guidance. I would like to thank Ville Nordström for providing a crucial insight, and many helpful discussions. I would like to thank Tobias Ekholm, John Etnyre, and Ian Zemke for helpful discussions. I would like to thank the referee for the careful read of this paper and for helpful comments.

\section{Background}\label{background}

\subsection{Legendrian contact homology}

In this section we review some background in Legendrian contact homology. We refer the reader to \cite{EES2} for details, and \cite{EENS} for another exposition on the treatment for disconnected Legendrian submanifolds.

Let \((M,d\lambda)\) be a \(2n\)-dimensional exact symplectic manifold with finite geometry at infinity (see \cite{EES2}*{Definition 2.1} for the notion of finite geometry at infinity). Let \(V=M\times\R\). Equip \(V\) with the standard contact form \(\alpha=dz-\lambda\), where \(z\) is the \(\R\) component (\(V\) is sometimes called the \textit{contactization} of \(M\)). Let \(\Pi:V\to M\) denote the Lagrangian projection.

With the contact form \(\alpha\), the Reeb vector field \(R_\alpha\) satisfies \(R_\alpha=\partial_z\). So, given a Legendrian submanifold \(\Lambda\subset V\), the Reeb chords of \(\Lambda\) correspond to self intersections of \(\Pi(\Lambda)\). We can perturb \(\Lambda\) so that there are finitely many Reeb chords, and the corresponding self intersections of \(\Pi(\Lambda)\) are transverse double points (if the Legendrian \(\Lambda\) satisfies this condition, we call it \emph{chord generic}). We denote the set of Reeb chords by \(\mathcal{R}\).

Let \(\Lambda\subset V\) be a chord generic compact Legendrian submanifold with vanishing Maslov class. Let \(\mathcal{R}\) denote the set of Reeb chords on \(\Lambda\). We fix a point \(p_j\in\Lambda_j\) on each component \(\Lambda_j\) of \(\Lambda\). We pick one of the basepoints, say \(p_1\) for convenience, which we call the \emph{special basepoint}. For each Reeb chord endpoint in \(\Lambda_j\) we choose an \emph{endpoint path} connecting the endpoint to \(p_j\). Furthermore, for \(j=1,...,r\), we choose paths \(\gamma_{1j}\) in \(M\) connecting \(\Pi(p_1)\) to \(\Pi(p_j)\) and symplectic trivialisations of \(\gamma_{1j}^*TM\) in which the tangent space \(\Pi(T_{p_1}\Lambda)\) corresponds to the tangent space \(\Pi(T_{p_j})\); for \(j=1\), \(\gamma_{11}\) is the trivial path. For any \(1\leq i,j\leq r\), we can then define \(\gamma_{ij}\) to be the path \(\gamma_{1i}\cup \gamma_{1j}\) joining \(\Pi(p_i)\) to \(\Pi(p_j)\), and \(\gamma_{ij}^*TM\) inherits a symplectic trivialisation from the trivialisations for \(\gamma_{1i}\) and \(\gamma_{1j}\).

We write \(\bar{\pi}_1(\Lambda)\) for the free product of \(\pi_1(\Lambda_j,p_j)\) of all the path components \(\Lambda_j\) of \(\Lambda\).

\begin{defn}\label{defn of CEDGA}
	The \textit{Chekanov Eliashberg differential graded algebra}, \((CE(V,\Lambda),\partial)\) is defined as follows:
	\begin{enumerate}
		\item \textbf{Algebra}: \(CE(V,\Lambda)\) is the free non-commutative algebra over the group algebra \(\Z/2[\bar{\pi}_1(\Lambda)]\), generated by elements of \(\mathcal{R}\), where the coefficients do not commute with the generators. In other words, \(CE(V,\Lambda)\) is the free non-commutative algebra over \(\Z/2\) generated by elements of \(\mathcal{R}\), and \(\bar{\pi}_1(\Lambda)\), modulo relations of the form \(hg=h\cdot_{\bar{\pi}_1(\Lambda)}g\) for all \(h,g\in \bar{\pi}_1(\Lambda)\), and the relation that the identity of \(\bar{\pi}_1(\Lambda)\) is identified with \(1\).

		\item \textbf{Grading}: Elements of \(\bar{\pi}_1(\Lambda)\) have degree 0. If \(q\) is a Reeb chord we define the grading by associating a path of Lagrangian subspaces to \(q\). If the endpoints of \(q\) lie on the same component \(\Lambda_j\), consider the path of tangent planes along the endpoint path from the final point of \(q\) to \(p_j\) followed by the reverse endpoint path from \(p_j\) to the initial point of \(q\). We then close this path to a loop \(\hat{\gamma}_q\) (see \cite{EES1} for details). In the case where the endpoints of \(q\) lie on different components, we associate a loop of Lagrangian subspaces \(\hat{\gamma}_q\) to \(q\) in the same way except that we insert the path of Lagrangian subspaces induced by the trivialisation along the chosen path, \(\gamma_{ij}\), connecting the components of the endpoints in order to connect the two paths from Reeb chord endpoints to chosen points. The grading of \(q\) is then \[|q|=\mu(\hat{\gamma}_q)-1,\] where \(\mu\) denotes the Maslov index. Note that the grading is independent of choices of endpoint paths in the case where \(\Lambda\) is connected, and the grading depends on the choices of paths in the case where \(\Lambda\) is disconnected.
		
		In the case where \(\Lambda\) is disconnected, and \(M\times\R=J^1(X)\) is a \(1\)-jet space, the choice of these paths is equivalent to choosing a Maslov potential, but we will not need that in this paper. A Maslov potential on the \(\Lambda\) assigns an integer to points on \(\Lambda\), in a way that is locally constant, except for at cusps of the front projection, where the integer increases by one when going up through a cusp.

		\item \textbf{Differential}: We define the differential on generators, and extend by the Leibniz rule. We set \(\partial \bar{\pi}_1(\Lambda)=0\). For a Reeb chord \(q_0\), \[\partial q_0=\sum_{{
			\begin{matrix}
				&\bar{A}=(A_0,\dots,A_k),\\
				&\sum_{j=1}^k|q_j|=|q_0|-1\\
			\end{matrix}}}
		|\mathcal{M}_{\bar{A}}(q_0;q_1...q_k)|A_0q_1A_1q_2...A_{k-1}q_kA_k,\] where \(A_i\in \bar{\pi}_1(\Lambda)\), and \(q_i\) are Reeb chords, and \(\mathcal{M}_{\bar{A}}(q_0;q_1,...,q_k)\) is the moduli space defined below.
	\end{enumerate}
\end{defn}

Now we define the moduli space \(\mathcal{M}_{\bar{A}}(q_0;q_1...q_k)\). Let \(J\) be an almost complex structure on the symplectisation \((V\times\R,d(e^t\alpha))\) of \(V\) (where \(\alpha\) is the contact form we equipped \(V\) with, and \(t\) is the \(\R\) coordinate) that is compatible with the symplectisation in the following sense: \(J\) is \(\R\)-invariant, \(J(\partial/\partial t)=R_\alpha\), and \(J\) maps \(\xi=\ker\alpha\) to itself. With respect to this almost complex structure, \(q_i \times \R\) is a holomorphic strip for any Reeb chord \(q_i\) of \(\Lambda\).

Let \(D_k^2=D^2\setminus \{p^+,p^-_1,...,p^-_k\}\) be a closed disk with \(k+1\) punctures on its boundary, labeled \(p^+,p^-_1,...,p^-_k\) in counterclockwise order around \(\partial D^2\).

For a \((k+1)\)-tuple of loops \(\bar{A}=(A_0,\dots,A_k)\) (more precisely, the loops are elements of \(\pi(\Lambda_j)\) for some \(j\)), we let \[\mathcal{M}_{\bar{A}}(q_0;q_1,...,q_k)\] denote the moduli space of \(J\)-holomorphic maps \[u:(D^2_k,\partial D^2_k)\to (V\times\R,\Lambda\times\R)\] up to domain reparametrisation and \(\R\)-translation, such that:
\begin{enumerate}
	\item near \(p^+\), \(u\) is asymptotic to a neighborhood of the Reeb strip \(q_0\times\R\) near \(t=\infty\);
	\item near \(p_l^-\) for \(1\leq l\leq k\), \(u\) is asymptotic to a neighborhood of \(q_l\times\R\) near \(t=-\infty\);
	\item when \((\Pi\circ u)|_{(p_j,p_{j+1})}\) is completed to a loop using the endpoint paths, it represents the class \(A_j\). Here \((p_j,p_{j+1})\) denotes the boundary interval in \(\partial D^2\) between \(p_j\) and \(p_{j+1}\).
\end{enumerate}

It is proved in \cite{EES2} that for generic admissible \(\Lambda\) and \(J\), \(\mathcal{M}_{\bar{A}}(q_0;q_1,...,q_k)\) is a transversely cut out manifold of dimension \(|q_0|-\sum_{i=1}^{k}|q_i|-1\). Roughly, admissibility here means that around the endpoints of Reeb chords, \(\Lambda\) is real analytic with respect to some standard coordinate given by \(J\), and generic admissible means that we further perturb \(\Lambda\) in the space of admissible Legendrian submanifolds (for the precise notion of generic admissible, see the discussion surrounding \cite{EES2}*{Section 2.3}). Furthermore, it admits a compactification as a manifold with boundary with corners in which the boundary consists of broken disks. Consequently, if the dimension equals \(0\) then it is compact.

Next we review the notion of stable tame isomorphism, first introduced by Chekanov in \cite{Che}. Given a semi-free DGA \(\mathcal{A}=R\langle q_1,\ldots,q_n\rangle\) (Definition \ref{semi-free dga}), an \emph{elementary automorphism} is an algebra automorphism of \(\mathcal{A}\) that sends a free generator \(q_i\) to \(xq_i y+u\) (here, \(x,y\in R\) are coefficients, and \(u\) is a linear combination of words not involving q), and sends all other generators to themselves. A \emph{tame automorphism} is a composition of elementary automorphisms. A \emph{tame isomorphism} between \((R\langle q_1,\ldots,q_n\rangle,\partial_1)\) and \((R\langle q'_1,\ldots,q'_n\rangle,\partial_2)\) is a chain map given by the composition of a tame automorphism and the map \(q_i\mapsto q'_i\) for all \(i\). A \emph{stabilisation} of \((R\langle q_1,\ldots,q_n\rangle,\partial)\) is \((R\langle q_1,\ldots,q_n,e_1,e_2\rangle,\partial)\) such that \(\partial(e_1)=e_2\), and \(\partial\) agrees with the original \(\partial\) on all other generators, and we also call the inclusion \((R\langle q_1,\ldots,q_n\rangle,\partial)\hookrightarrow (R\langle q_1,\ldots,q_n,e_1,e_2\rangle,\partial)\) a stablisation. A \emph{destabilisation} is the quotient map \((R\langle q_1,\ldots,q_n,e_1,e_2\rangle,\partial)\twoheadrightarrow (R\langle q_1,\ldots,q_n\rangle,\partial)\). A \emph{stable tame isomorphism} is a composition of tame isomorphisms, stabilisations, and destabilisations.

\begin{thm}[\cite{EES2}]
	The map \(\partial:CE(V,\Lambda)\to CE(V,\Lambda)\) is a differential, that is \(\partial^2=0\). The stable tame isomorphism class of \(CE(V,\Lambda)\) is an invariant of \(\Lambda\) up to Legendrian isotopy.
\end{thm}

\begin{rem}
	In \cite{EES2}, they proved the version where we replace \(\bar{\pi}_1(\Lambda)\) by \(H_1(\Lambda)\), and where homology classes commute with Reeb chords. The proofs carry over to our setting with no change.
\end{rem}

\subsection{Quantum flow trees}

In this section, we review the basics of quantum flow trees. For details, see \cite{EL}*{section 3.3.3} and \cite{EENS}*{sections 4,5}. The following discussion of quantum flow trees is a simplified version, where we only concern ourselves with partial flow lines. A version of this first appeared in \cite{EESu}.

\begin{defn}
	Let \(f:X\to\R\) be a Morse function on a closed Riemannian manifold \(X\). A \textit{partial flow line} of \(f\) is a smooth curve \(\Delta:[0,\infty)\to X\) that satisfies \[\dot{\Delta}(t)=-\nabla(f)(\Delta(t))\] on \(t\in(0,\infty)\). We say that \(\Delta(0)\) is the \textit{special puncture} of \(\Delta\). A \textit{complete flow line} of \(f\) is a smooth curve \(\Delta:\R\to X\) satisfying the same flow equation. A complete flow line does not have a special puncture.
\end{defn}

Since a partial flow line is determined by its special puncture, we have that the moduli space of partial flow lines \(\Delta\) with \(\lim_{t\to \infty}\Delta(t)\) fixed is the stable manifold of \(\lim_{t\to\infty}\Delta(t)\), which means that the formal dimension of the moduli space at a partial flow line \(\Delta\) is \[\dim(\Delta)=\dim(X)-\ind\left(\lim_{t\to\infty}\Delta(t)\right).\]

The dimension of the moduli space of complete flow lines \(\Delta\), with \(\lim_{t\to\pm \infty}\Delta(t)\) fixed, modulo the translational \(\R\)-action, is \[\dim(\Delta)=\ind\left(\lim_{t\to-\infty}\Delta(t)\right)-\ind\left(\lim_{t\to\infty}\Delta(t)\right)-1.\]

\begin{defn}
	Let \(V=J^1(X)=T^*X\times\R\) be a 1-jet space. A \textit{graphical Legendrian submanifold} is a Legendrian submanifold of the form \(\Gamma_{df,X}=\{\left(x,d_x f,f(x)\right)\}\) for some smooth function \(f:X\to\R\).
\end{defn}

Let \((M,d\lambda)\) be a \(2n\)-dimensional exact symplectic manifold with finite geometry at infinity. Let \(V=M\times\R\). Equip \(V\) with the standard contact form \(\alpha=dz-\lambda\), where \(z\) is the \(\R\) component. Let \(\Lambda\subset V\) be a Legendrian submanifold. Let \(U_\Lambda\subset V\) be a small contact neighborhood of \(\Lambda\). By the Legendrian neighborhood theorem, we identify \(U_\Lambda\) with a neighborhood of the 0-section of \(J^1\Lambda\). Let \(f_1,f_2:\Lambda\to\R\) be small smooth functions such that \(f_1<f_2\). Let \(\Lambda'=\Gamma_{df_1,\Lambda}\cup\Gamma_{df_2,\Lambda}\subset U_\Lambda\subset V\). Note that \(\Gamma_{df_1,\Lambda}\) and \(\Gamma_{df_2,\Lambda}\) are disjoint. Choose a Riemannian metric \(g\) on \(\Lambda\). Perturb the \(f_i,g\) so that \(\Lambda'\) is quantum flow tree-generic, in the sense of \cite{EL}*{Section 3.3.3}, which implies that \(\Lambda'\) is generic admissible in the sense of \cite{EES2}*{Sections 2.3}.

Roughly, the data of an almost complex structure, with \(\Lambda\), with \(f_2-f_1\) is quantum flow tree-generic in the sense of \cite{EL}*{Section 3.3.3}, if:
\begin{itemize}
\item the moduli spaces of Morse flow lines of \(f_2-f_1\) are transversely cut out,
\item the almost complex structure with \(\Lambda\) is generic admissible,
\item partial flow trees are transverse to the boundary evaluation maps of
the transversely cut out holomorphic curves.
\end{itemize}

\begin{defn}
	Let \(\Delta\) be a (partial or complete) flow line of \(-\nabla(f_2-f_1)\). The \(1\)\textit{-jet lift} of \(\Delta\) is \(\pi^{-1}(\Ima(\Delta))\), where \(\pi:\Lambda'\to\Lambda\) is the trivial 2-fold covering map obtained by restricting the base projection \(\pi:U_\Lambda\to\Lambda\).
\end{defn}

\begin{defn}
	A \textit{quantum flow tree} $u_{\Delta}$ on \(\Lambda'\) consists of a punctured holomorphic disk \(u:(D_k^2,\partial D_k^2)\to (V\times\R,\Lambda\times\R)\) with 1 positive puncture satisfying the same asymptotic conditions as the holomorphic disks in the differential of the Chekanov-Eliashberg DGAs, and a finite set of partial flow lines \(\Delta=\left\{\Delta_j\right\}_{j=1}^m\) of \(-\nabla(f_2-f_1)\) each with a special puncture \(x_j\), such that $x_j \in \partial \tilde{u}$ for all \(j\). Here, \(\tilde{u}:(D_k^2,\partial D_k^2)\to (V\times\R,\Lambda\times\R)\to (V,\Lambda)\) is given by \(u\) composed with the projection map. Then the special punctures subdivide the boundary \(\partial \tilde{u}\) into arcs on \(\Lambda\). For these arcs on \(\Lambda\), we pick a \emph{\(1\)-jet lift} to \(\Lambda'\). More specifically, a \(1\)-jet lift consists of arcs on \(\Lambda'\), and the projection map \(\pi:\Lambda'\to\Lambda\) restricted to the \(1\)-jet lift is a homeomorphism onto the original arcs on \(\Lambda\). The 1-jet lift of these arcs, together with the 1-jet lifts of the flow lines in \(\Delta\), are required to form a closed curve when projected to \(M\). We say that the \emph{1-jet lift of a quantum flow tree} is the union of these arcs and the 1-jet lifts of the flow lines in \(\Delta\). The 1-jet lift of a quantum flow tree is considered part of the information of a quantum flow tree. A complete flow line of \(-\nabla(f_2-f_1)\) on \(\Lambda\) is also considered a quantum flow tree.
\end{defn}

By intersecting the moduli spaces of special punctures (which are isomorphic to the moduli spaces of the partial flow lines) with the boundary arcs of the moduli space of \(u\), we see that the formal dimension of a quantum flow tree \(u_\Delta\) (that is not a complete flow line), up to domain reparametrisation and \(\R\)-translation of \(u\), is
\begin{equation}\label{quantum flow tree dimension}
\dim(u_{\Delta})=\dim(u)+\sum_{\Delta_{j}\in\Delta}(\dim(\Delta_{j})-n+1),
\end{equation}
where \(\dim(u)\) is the dimension of the moduli spaces of holomorphic disks with the same boundary conditions as \(u\), up to domain reparametrisation and \(\R\)-translation. A version of the above formula for the \(n=2\) case appears in \cite{EENS}*{Section 4.1}, for multiscale flow trees, which are similar and related to quantum flow trees. The formal dimension of a complete flow line \(\Delta\) up to the \(\R\)-translation is still \[\dim(\Delta)=\ind\left(\lim_{t\to-\infty}\Delta(t)\right)-\ind\left(\lim_{t\to\infty}\Delta(t)\right)-1.\]

We say that a quantum flow tree \(u_{\Delta}\) is \textit{rigid} if \(\dim(u_{\Delta})=0\) and if it is transversely cut out by its defining equations.

\begin{thm}[\cite{EESu},\cite{EENS},\cite{EL}]\label{quantum flow tree}
	If the neighborhood \(U_\Lambda\) is contained in a small neighborhood of \(\Lambda\), then for any \(\epsilon>0\), there exists an almost complex structure \(J\) on \(V\times\R\), compatible with the symplectisation, and regular with respect to holomorphic disks with one positive puncture of dimension \(\leq 1\), such that there is a one-to-one correspondence between:
	\begin{itemize}
	\item rigid holomorphic disks \((D_k^2,\partial D_k^2)\to (V\times\R,\Lambda'\times\R)\) with one positive puncture satisfying the same asymptotic conditions as the holomorphic disks in the differential of the Chekanov-Eliashberg DGAs,
	\item and rigid quantum flow trees on \(\Lambda'\). 
	\end{itemize}
	Furthermore, the \(1\)-jet lift of a quantum flow tree lies in an \(\epsilon\)-neighborhood of the boundary of the corresponding holomorphic disk (after we project the boundary of the corresponding holomorphic disk to \(\Lambda\)).
\end{thm}

Therefore, there is some almost complex structure, using which, counting rigid holomorphic disks with boundary on \(\Lambda'\) is equivalent to counting rigid quantum flow trees on \(\Lambda'\). Such an almost complex structure is said to be compatible with the neighborhood \(U_\Lambda\).

Next we record some facts about rigid disks on \(\Lambda'\) (equivalently quantum flow trees) with a negative puncture at a minimum of \(f_2-f_1\).

\begin{defn}
	A Reeb chord of \(\Lambda'\) is called a \textit{short Reeb chord} if it is contained in \(U_\Lambda\), and is called a \textit{long Reeb chord} otherwise.
\end{defn}

Note that short Reeb chords of \(\Lambda\) correspond to critical points of the function \(f_2-f_1\).

\begin{lem}\label{disks to a minimum in 2 sheets}
	A rigid holomorphic disk with a negative puncture at a short Reeb chord corresponding to a minimum (index 0 critical point) of \(f_2-f_1\) corresponds to either:
	\begin{enumerate}
		\item a rigid complete flow line of \(f_2-f_1\), or
		\item \(u_\Delta\) where \(u\) is a constant disk at some Reeb chord \(b\) of \(\Lambda\), and \(\Delta\) consists of a single partial flow line, flowing from an endpoint of \(b\) to a minimum of \(f_2-f_1\). Here, a constant disk means a holomorphic strip (that is, a holomorphic disk with 1 positive puncture and 1 negative puncture) that equals \(b\times\R\).
	\end{enumerate}
\end{lem}

\begin{proof}
	By definition, a rigid quantum flow tree can be a rigid complete flow line. If we assume that a quantum flow tree \(u_\Delta\) is not a rigid complete flow line, and that \(u\) is not a constant disk, then the special puncture on the partial flow line flowing to the minimum can move along \(\partial u\), making \(u_\Delta\) a quantum flow tree that is not rigid.
\end{proof}

Next we let \(f_1<f_2<f_3:\Lambda\to \R\) be small smooth functions. Perturb the \(f_i\) so that \(\Lambda'\) is quantum flow tree-generic, in the sense of \cite{EL}*{section 3.3.3}. Assume that \(\left\|f_2-f_1\right\|_{C^2}<\epsilon\) for some small \(\epsilon\). Let \(\Lambda''=\Gamma_{df_1,\Lambda}\cup\Gamma_{df_2,\Lambda}\cup\Gamma_{df_3,\Lambda}\).

We note that rigid holomorphic disks on \(\Lambda''\) correspond to quantum flow trees \(u_\Delta\) with disk \(u\) with boundary on \(\Lambda'\) and \(\Delta\) a set of partial flow lines of \(f_2-f_1\) with the appropriate lifts. To see this, let \(f_{2.5}:\Lambda\to\R\) be a smooth function such that \(f_1<f_2<f_{2.5}<f_3\), and that \(\left\|f_3-f_{2.5}\right\|_{C^2}<\epsilon\). Let 
\[
\Lambda'''=\Gamma_{df_1,\Lambda}\cup\Gamma_{df_2,\Lambda}\cup\Gamma_{df_{2.5},\Lambda}\cup\Gamma_{df_3,\Lambda}.
\]
Perturb the \(f_i\) so that \(\Lambda'''\) is quantum flow tree-generic. Then, in the setting of Theorem \ref{quantum flow tree}, take \(\Lambda=\Gamma_{df_2,\Lambda}\cup\Gamma_{df_3,\Lambda}\), and we see that \(\Lambda'''\) is in a small neighborhood of \(\Gamma_{df_2,\Lambda}\cup\Gamma_{df_3,\Lambda}\), and the projection \(\Lambda'''\to \Gamma_{df_2,\Lambda}\cup\Gamma_{df_3,\Lambda}\) is a trivial 2-fold cover. Applying Theorem \ref{quantum flow tree}, and discarding the quantum flow trees with boundary lifts involving \(\Gamma_{f_{2.5},\Lambda}\), we see the correspondence. By Lemma \ref{disks to a minimum in 2 sheets}, we also have the following:

\begin{cor}\label{disks to a minimum in 3 sheets}
	Rigid holomorphic disks on \(\Lambda''\) with a negative puncture at a short Reeb chord corresponding to a minimum of \(f_2-f_1\) correspond to either:
	\begin{enumerate}
		\item a rigid complete flow line of \(f_2-f_1\), or
		\item a quantum flow tree \(u_\Delta\) where \(u\) is a constant disk at some Reeb chord \(b\) of \(\Lambda'\), and \(\Delta\) consists of a single partial flow line, flowing from an endpoint of \(b\) to a minimum of \(f_2-f_1\).
	\end{enumerate}
\end{cor}

Since Corollary \ref{disks to a minimum in 3 sheets} allows us to compute some holomorphic disks on \(\Lambda''=\Gamma_{df_1,\Lambda}\cup\Gamma_{df_2,\Lambda}\cup\Gamma_{df_3,\Lambda}\) by considering holomorphic disks on \(\Gamma_{df_2,\Lambda}\cup\Gamma_{df_3,\Lambda}\) and flow lines of \(f_2-f_1\), we say that we \emph{degenerate} \(\Gamma_{df_1,\Lambda}\cup\Gamma_{df_2,\Lambda}\cup\Gamma_{df_3,\Lambda}\) onto \(\Gamma_{df_2,\Lambda}\cup\Gamma_{df_3,\Lambda}\).

\subsection{\(A_\infty\)-categories and augmented DGAs}\label{section augmented DGAs}

We work with cohomologically unital \(\Z\)-graded \(A_\infty\)-categories over a field \(k\) of characteristic 2. An \(A_\infty\)-equivalence is an \(A_\infty\)-functor whose cohomology functor is an equivalence between the cohomology categories. Since we work over a field of characteristic 2, we ignore all the signs for simplicity. We use the ``forward composition'' convention:

Let \(\mathcal{C}\) be an \(A_\infty\)-category, and let \(X,Y\) be objects. Let \(\mathcal{C}(X,Y)\) denote the hom space. We write composition as \[m_k:\mathcal{C}(X_{0},X_{1})\otimes \mathcal{C}(X_{1},X_{2})\otimes \cdots \otimes \mathcal{C}(X_{k-1},X_{k})\to \mathcal{C}(X_{0},X_{k}).\]

For the following construction of an \(A_\infty\)-algebra from an augmented DGA, see \cite{NRSSZ}*{section 3.1} for details.

\begin{defn}\label{semi-free dga}
	A \emph{semi-free} DGA is a DGA that as an algebra, is the free non-commutative algebra generated by a set \(\mathcal{R}\), over a non-commutative coefficient ring \(R\), such that elements of the coefficient ring \(R\) do not commute with the generators.
	
	In other words: let \(R\) be a non-commutative coefficient ring, with a generating set \(\mathcal{H}\), and some relations. A semi-free DGA is a DGA \(\mathcal{A}\) equipped with a set \(\mathcal{S}=\mathcal{R}\cup\mathcal{H}\) of generators, such that \(\mathcal{A}\) is the result of taking the free non-commutative unital algebra generated by the elements of \(\mathcal{S}\), and quotienting by relations on \(\mathcal{H}\) that are inherited from the relations that define \(R\). In addition, for any \(h\in\mathcal{H}\), we have \(|h|=0\), and \(\partial h=0\).
\end{defn}

\begin{rem}
Examples of semi-free DGAs as defined above include Chekanov-Eliashberg DGAs for a pair \((V,\Lambda)\), setting \(\mathcal{R}\) to be the set of Reeb chords, and \(\mathcal{H}\) is the generating set of \(\bar{\pi}(\Lambda)\) given by homotopically nontrivial based loops on each component of \(\Lambda\).
\end{rem}

Let \(\mathcal{A}\) be a semi-free DGA. Let \(k\) be a field of characteristic 2. A \emph{\(k\)-augmentation} of \(\mathcal{A}\) is a DGA map \(\epsilon:\mathcal{A}\to k\), where \(k\) is considered as a DGA concentrated in degree 0 with trivial differential.

Given an augmentation \(\epsilon:\mathcal{A}\to k\), we define the \(k\)-algebra \[\mathcal{A}^\epsilon:=(\mathcal{A}\otimes k)/(h=\epsilon(h)\text{, for }h\in \mathcal{H}).\] Since \(\partial h=0\) for any \(h\in \mathcal{H}\), we have that \(\partial\) descends to \(\mathcal{A}^\epsilon\).

We write \(C\) for the free \(k\)-module with basis \(\mathcal{R}\). We have \[\mathcal{A}^\epsilon=\bigoplus_{k\geq 0}C^{\otimes k};\] and we further define \(\mathcal{A}^\epsilon_+\subset \mathcal{A}^\epsilon\) by \[\mathcal{A}^\epsilon_+:=\bigoplus_{k\geq 1}C^{\otimes k}.\]

\(\partial\) does not necessarily preserve \(\mathcal{A}^\epsilon_+\). Consider the \(k\)-algebra automorphism \(\phi_\epsilon:\mathcal{A}^\epsilon\to\mathcal{A}^\epsilon\), determined by \(\phi_\epsilon(q)=q+\epsilon(q)\) for \(q\in\mathcal{R}\). Conjugating by this automorphism gives rise to a new differential \[\partial_\epsilon:=\phi_\epsilon\circ\partial\circ\phi_\epsilon^{-1}:\mathcal{A}^\epsilon\to \mathcal{A}^\epsilon,\] which preserves \(\mathcal{A}^\epsilon_+\), and in particular, descends to a differential on \(\mathcal{A}^\epsilon_+/(\mathcal{A}^\epsilon_+)^2\cong C\).

The basis \(\mathcal{R}=\{q_i\}\) for \(C\) gives a dual basis \(\{q_i^*\}\) for the dual space \(C^*\) with \(\langle q_i^*,q_j \rangle=\delta_{ij}\), and we grade \(C^*\) by \(|q_i^*|=|q_i|\).

For a \(k\)-module \(V\), we write \(T(V):=\oplus_{n\geq0}V^{\otimes n}\) for the tensor algebra, and \(\bar{T}(V):=\oplus_{n\geq1}V^{\otimes n}\). The pairing extends to a pairing between \(T(C^*)\) and \(T(C)\) determined by \[(q_{i_1}...q_{i_k})^*=q_{i_k}^*...q_{i_1}^*,\] equivalently, \[\langle q_{i_k}^*...q_{i_1}^*,q_{i_1}...q_{i_k} \rangle=1\] and all other pairings are 0. On the positive part \(\bar{T}(C^*)\) of the tensor algebra \(T(C^*)\), we define \(\partial_\epsilon^*\) to be the co-differential dual to \(\partial_\epsilon\): \[\langle\partial_\epsilon^* x,y\rangle=\langle x,\partial_\epsilon y\rangle.\]

Shift gradings by defining \(C^\vee:=C^*[-1]\); then \(\bar{T}(C^*)=\bar{T}(C^\vee[1])\). By \cite{NRSSZ}*{Proposition 2.7} (which comes from \cite{Sta} and \cite{Kad}), the co-differential \(\partial_\epsilon^*\) now determines an \(A_\infty\)-structure on \(C^\vee\). We write the corresponding \(A_\infty\)-multiplications as \[m_k(\epsilon):(C^\vee)^{\otimes k}\to C^\vee.\]

Let \(s:C^\vee\to C^\vee[1]=C^*\) be the degree \(-1\) suspension map. For \(q\in\mathcal{R}\), write \(q^\vee:=s^{-1}(q^*)\). Concretely, the \(A_\infty\)-multiplication is given by \[m_k(\epsilon)(q_{i_1}^\vee,...,q_{i_k}^\vee)=\sum_{q\in\mathcal{R}} q^\vee\cdot \mathrm{Coeff}_{q_{i_k}...q_{i_1}}(\partial_\epsilon q).\]

\subsection{Link grading}

In this section we review link gradings, introduced in \cite{Mis}. Our treatment is a minor variation of \cite{NRSSZ}*{section 3.2}, where we relabel the set \(\{1,...,m\}\) by a finite set \(P\). The definitions and arguments all carry through with no change.

\begin{defn}
	Let \(\mathcal{A}\) be a semi-free DGA with generating set \(\mathcal{S}=\mathcal{R}\cup\mathcal{H}\). Let \(P\) be a finite set. A \(P\)-\textit{link grading} on \(\mathcal{A}\) is a choice of a pair of maps \[c,r:\mathcal{S}\to P\] satisfying the following conditions:
	\begin{enumerate}
		\item for any \(x\in\mathcal{S}\) with \(r(x)\neq c(x)\), each term in \(\partial x\) is a word of the form \(x_1...x_k\) where \(c(x_{i})=r(x_{i+1})\) for \(i=1,...,k-1\) and \(c(x_k)=c(x),r(x_1)=r(x)\) (such a word is called \textit{composable});
		\item for any \(x\in\mathcal{S}\) with \(c(x)=r(x)\), each term in \(\partial x\) is either a composable word, or 1;
		\item for any \(h\in\mathcal{H}\), we have \(c(h)=r(h)\).
	\end{enumerate}
\end{defn}

For \(i,j\in P\), we write \(\mathcal{S}^{ij}:=(c\times r)^{-1}(i,j)\), and likewise \(\mathcal{R}^{ij}\) and \(\mathcal{T}^{ij}\). We call elements of \(\mathcal{S}^{ii}\) \textit{diagonal} and elements of \(\mathcal{S}^{ij}\) for \(i\neq j\) \textit{off-diagonal}. Note that all elements of \(\mathcal{H}\) are diagonal.

\begin{rem}\label{Legendrian link grading}
	If \(\Lambda=\coprod_P \Lambda^i\) is a disjoint union of Legendrians indexed by a finite set \(P\), then the DGA for \((V,\Lambda)\) has a \(P\)-link grading: for each Reeb chord \(q\), define \(c(q)\) (respectively \(r(q)\)) to be the index of the component containing the beginning point (respectively endpoint) of \(q\), and for \(\gamma\in \pi_1(\Lambda^i,p_i)\), define \(r(A)=c(A)=i\).
\end{rem}

For the remainder of this subsection, fix a finite set \(P\), and a semi-free DGA \(\mathcal{A}\) with a \(P\)-link grading. For proofs of the following claims, see \cite{NRSSZ}*{section 3.2}.

\begin{prop}
	If \(\pi:P=\sqcup_i P_i\) is any partition, let \(J_\pi\) be the two-sided ideal generated by all elements \(a\) with \(r(a),c(a)\) in different parts. Then \(J_\pi\) is preserved by \(\partial\).
\end{prop}

Note that \(\mathcal{A}/J_\pi\) remains a semi-free algebra with generators \(\mathcal{H}\) and some subset of \(\mathcal{R}\); it moreover inherits the link grading.

\begin{defn}\label{Api definition}
	For a partition \(\pi\) of \(P\), we write \(\mathcal{A}_\pi:=\mathcal{A}/J_\pi\). In the special case where \(\pi=I\sqcup I^c\) for some \(I\subset P\), we write \(\mathcal{A}_I\) for the subalgebra of \(\mathcal{A}_{I\sqcup I^c}\) generated by the elements of \(\coprod_{i,j\in I}\mathcal{S}^{ij}\). Finally, we will write \(\mathcal{A}_i:=\mathcal{A}_{\{i\}}\).
\end{defn}

\begin{prop}
	For any \(I\subset P\), the algebra \(\mathcal{A}_I\) is preserved by the differential inherited by \(\mathcal{A}_{I\sqcup I^c}\).
\end{prop}

\begin{prop}
	For any partition \(\pi:P=\sqcup_i P_i\), we have \(\mathcal{A}_\pi=\star_i \mathcal{A}_{P_i}\) (here \(\star\) is the free product of DGAs).
\end{prop}

In particular, an augmentation of \(\mathcal{A}\) which annihilates generators \(a\) with \(c(a),r(a)\) in different parts is the same as a tuple of augmentations of the \(\mathcal{A}_{P_i}\).

\begin{defn}
An augmentation \(\epsilon:\mathcal{A}\to k\) is said to \textit{respect the link grading} if \(\epsilon\) sends off-diagonal elements of \(\mathcal{R}\) to \(0\).
\end{defn}

Let \(\epsilon:\mathcal{A}\to k\) be an augmentation. We write \(C^{ij}\) for the free \(k\)-submodule of \(C\) generated by \(\mathcal{R}^{ij}\), so that \(C=\oplus_{i.j\in P}C^{ij}\). Similarly, we split \(C^\vee=\oplus C_{ij}^\vee\). The product then splits into terms \[m_k(\epsilon):C^\vee_{i_1 j_1}\otimes C^\vee_{i_2 j_2}\otimes\cdots \otimes C^\vee_{i_k j_k}\to C^\vee_{ij}.\]

\begin{prop}
	Assume \(\epsilon\) respects the link grading. Then the product \(m_k(\epsilon):C^\vee_{i_1 j_1}\otimes C^\vee_{i_2 j_2}\otimes\cdots \otimes C^\vee_{i_k j_k}\to C^\vee_{ij}\) vanishes unless \(i_1=i,j_k=j,i_{r+1}=j_{r}\).
\end{prop}

That is, the nonvanishing products are:
\begin{equation}\label{A-infinity operations}
m_k(\epsilon):C^\vee_{i_1 i_{2}}\otimes\cdots\otimes C^\vee_{i_{k} i_{k+1}}\to C^\vee_{i_1 i_{k+1}}.
\end{equation}

\begin{prop}\label{P category}
	Let \(\epsilon:\mathcal{A}\to k\) be an augmentation that respects the \(P\)-link grading. Then, there is a (possibly non-unital) \(A_\infty\) category whose objects are elements of \(P\), with morphisms \(\hom(i,j)=C_{ij}^\vee\), and multiplications \(m_k(\epsilon)\) as above.
\end{prop}

\section{Construction and invariance of the augmentation category}\label{construction and invariance}

In \cite{NRSSZ}, augmentation categories are constructed by using a consistent sequence of perturbations. Since finding a consistent sequence of perturbations in higher dimensions is difficult, we use a different approach, by choosing a set of arbitrary generic perturbations with no consistency in mind, and then localise (roughly following the construction of wrapped Fukaya categories due to Abouzaid-Seidel, explained in \cite{GPS1}*{chapter 3}).

An advantage of using the localisation approach is that we get a more general invariance statement than using specific consistent sequences: suppose we find one family of consistent sequences of perturbations, and prove the invariance of the augmentation category obtained. If we find another family of consistent sequences, we do not automatically know if the new family of consistent sequences give rise to an equivalent augmentation category. Using the localisation approach, we prove that any arbitrary generic perturbations give rise to the same augmentation category, up to equivalence.

\subsection{An auxiliary category}

First, we build an auxiliary category (Definition \ref{defn of preaugmentation category}), and prove its invariance (Proposition \ref{oinv}). The auxiliary category is constructed by considering countably many generic positive pushoffs \(\Lambda^i\) of \(\Lambda\). The objects of the auxiliary category are augmentations of some \(\Lambda^i\), and the morphisms only ``go up'' in the index \(i\). This has the advantage that we can perturb the infinitely many copies of \(\Lambda\) generically, and obtain transversality of moduli spaces of holomorphic disks for any finite subset of them for free, by citing \cite{EES2}. We start with some algebraic constructions.

\begin{defn}
	Let \(\{\mathcal{A}^P\}_{P\subset\N}\) be a set of semi-free DGAs \(\mathcal{A}^P\) with generating sets \(\mathcal{S}^P=\mathcal{R}^P\cup\mathcal{H}^P\), indexed by all finite subsets \(P\) of \(\N\). Assume each \(\mathcal{A}^P\) is equipped with a \(P\)-link grading. We say that \(\{\mathcal{A}^P\}_{P\subset\N}\) is a \textit{system of semi-free DGAs} if:
	\begin{enumerate}
		\item The assignment \(P\mapsto \mathcal{S}^P\) is functorial, from the poset category of finite subsets of \(\N\) to the category of sets.
		\item For every inclusion \(P\subset P'\), the map \(\mathcal{S}^P\to\mathcal{S}^{P'}\) is an injection, and that it is compatible with the link gradings. That is, the following diagram commutes:
		\[\begin{tikzcd}[column sep=small,row sep=tiny]
		{\mathcal{S}^P} && P \\
		\\
		{\mathcal{S}^{P'}} && {P'.}
		\arrow["{c,r}", from=1-1, to=1-3]
		\arrow[from=1-1, to=3-1]
		\arrow[hook', from=1-3, to=3-3]
		\arrow["{c,r}", from=3-1, to=3-3]
		\end{tikzcd}\]
		\item By (1) and (2), every inclusion \(P\subset P'\) induces a map \(\mathcal{S}^P\to\mathcal{S}^{P'}\). The map \(\mathcal{S}^P\to\mathcal{S}^{P'}\) further induces a well-defined algebra homomorphism \(\mathcal{A}^P\to\mathcal{A}^{P'}_{P}\) (see \ref{Api definition} for the definition of \(\mathcal{A}^{P'}_{P}\)). We require this map \(\mathcal{A}^P\to\mathcal{A}^{P'}_{P}\) to be a DGA isomorphism.
	\end{enumerate}
\end{defn}

\begin{defn}
	Given a system of semi-free DGAs \(\{\mathcal{A}^P\}_{P\subset\N}\), and a field \(k\) of characteristic 2, define a strictly unital \(A_\infty\)-category \(\mathcal{O}_{\mathcal{A}}\) as follows:
	\begin{enumerate}
		\item The objects are pairs \((i,\epsilon)\) where \(\epsilon\) is an augmentation \(A^{\{i\}}\to k\),
		\item For \(i< j\), the morphisms are \[\mathcal{O}_{\mathcal{A}}((i,\epsilon),(j,\epsilon')):=C_{ij}^{\vee}.\] For \(i=j,\epsilon=\epsilon'\), we set the morphism \[\mathcal{O}_{\mathcal{A}}((i,\epsilon),(j,\epsilon')):=k.\] For all other choices of \((i,\epsilon),(j,\epsilon')\), we set \[\mathcal{O}_{\mathcal{A}}((i,\epsilon),(j,\epsilon')):=0.\]
		\item For \(k\geq 1\) and \(i_1<\cdots <i_{k+1}\), the composition map \[m_k:\mathcal{O}_{\mathcal{A}}((i_{1},\epsilon_{1}),(i_{2},\epsilon_{2}))\otimes \cdots\otimes\mathcal{O}_{\mathcal{A}}((i_{k},\epsilon_{k}),(i_{k+1},\epsilon_{k+1}))\to \mathcal{O}_{\mathcal{A}}((i_{1},\epsilon_{1}),(i_{k+1},\epsilon_{k+1}))\] is defined to be the map \ref{A-infinity operations}: \[m_{k}(\epsilon):C_{i_{1}i_{2}}^{\vee}\otimes\cdots\otimes C_{i_{k}i_{k+1}}^{\vee}\to C_{i_{1}i_{k+1}}^{\vee},\] where we take the DGA to be \(\mathcal{A}^{\{i_1,...,i_{k+1}\}}\), and take the augmentation \(\epsilon\) to be the diagonal augmentation \((\epsilon_{1},...,\epsilon_{k+1})\). In the case where some \(i_l=i_{l+1}\), \(m_k\) is formally fixed by strict unitality.
	\end{enumerate}
\end{defn}

The \(A_\infty\)-relations hold for the above construction, since they can be checked in the \(A_\infty\)-category from Proposition \ref{P category}.

Let \((M,d\lambda)\) be a \(2n\)-dimensional exact symplectic manifold with finite geometry at infinity. Let \(V=M\times\R\). Equip \(V\) with the standard contact form \(\alpha=dz-\lambda\), where \(z\) is the \(\R\) component. Let \(\Lambda\subset V\) be a compact Legendrian submanifold with vanishing Maslov class. Let \(\Lambda\subset U_\Lambda\subset V\) be a small contact neighborhood of \(\Lambda\). We require that the contact neighborhood \(U_\Lambda\) is small enough so that Theorem \ref{quantum flow tree} applies. Pick a sequence of increasing small functions \(f_{1}<f_{2}<\cdots\), and let \(\Lambda^{i}:=\Gamma_{df_i,\Lambda}\subset U_\Lambda \subset V\). Also pick an almost complex structure \(J\) (guaranteed by Theorem \ref{quantum flow tree}) such that counting rigid holomorphic disks is equivalent to counting rigid quantum flow trees.

On \(\Lambda\), we pick a set of basepoints \(B=\{p_1,\ldots,p_r\}\), which contains one basepoint for each component of \(\Lambda\), and one of the basepoints \(p_1\) is designated to be the special basepoint. As in the setting before Definition \ref{defn of CEDGA}, we pick paths \(\gamma_{1j}\) connecting each \(\Pi(p_j)\) to \(\Pi(p_1)\), and a symplectic trivialisation along these paths that identify the tangent spaces of their endpoints. On \(\Lambda^i\), pick \(p_j^i\) to be a basepoint that is identified with \(p_j\) on \(\Lambda\) via the base projection map \(\Lambda^{i}\hookrightarrow U_\Lambda\to\Lambda\). On \(\coprod_{i\in \N} \Lambda^{i}\), we designate \(p_{1}^1\) to be the special basepoint. We choose the curves that connect the basepoints to the special basepoint in the following way: connect all \(\Pi(p_1^i)\) to \(\Pi(p_1^1)\) via paths that are \(C^0\)-small. For another basepoint of \(p_i^j\), we connect \(\Pi(p_i^j)\) to \(\Pi(p_1^j)\) by a path whose corresponding path of Lagrangian subspaces is \(C^0\)-close to the path of Lagrangian subspaces corresponding to \(\gamma_{1j}\), and concatenate with the path from \(\Pi(p_1^j)\) to \(\Pi(p_1^1)\) that we picked above.

Reeb chords of \(\coprod_{i\in \N} \Lambda^{i}\) are partitioned into long Reeb chords (which are not contained in \(U_\Lambda\)) and short Reeb chords (which are contained in \(U_\Lambda\)). Long Reeb chords (and their endpoints) from \(\Lambda^i\) to \(\Lambda^j\) correspond to Reeb chords (and their endpoints) on \(\Lambda\), and the base projection of the endpoints of a long Reeb chord from \(\Lambda^i\) to \(\Lambda^j\) onto \(\Lambda\) are close to the corresponding endpoints of the Reeb chord on \(\Lambda\). Choose endpoint paths so that for an endpoint of a Reeb chord on \(\Lambda\), all the endpoint paths of the corresponding long Reeb chord endpoints are close to each other after base projection. Short Reeb chords from \(\Lambda^i\) to \(\Lambda^j\) (for \(i,j\)) correspond to critical points of \(f_j-f_i\). Choose endpoint paths of a short Reeb chord from \(\Lambda^i\) to \(\Lambda^j\) so that the path for the beginning point and the path for the ending point project to the same path on \(\Lambda\).

For generically chosen functions \(f_i\) and compatible almost complex structure \(J\), we have that for any finite set \(P\subset\N\), the Legendrian \(\Lambda^{P}:=\coprod_{i\in S} \Lambda^{i}\) is quantum flow tree-generic, and so gives rise to a Chekanov-Eliashberg DGA which we denote by \(\mathcal{A}^{P}\). By remark \ref{Legendrian link grading}, \(\mathcal{A}^{P}\) is equipped with a \(P\)-link grading. This gives rise to a system of semi-free DGAs \(\{\mathcal{A}^P\}_{P\subset\N}\).

\begin{defn}\label{defn of preaugmentation category}
We define the \textit{pre-augmentation category} \(\mathcal{O}_{\Lambda,f_i,J}\) to be \(\mathcal{O}_{\mathcal{A}}\), where \(\mathcal{A}\) is the system of semi-free of DGAs \(\{\mathcal{A}^P\}_{P\subset\N}\) coming from the parallel copy construction above.
\end{defn}

\begin{rem}
	Without vanishing Maslov class, the construction still goes through, with \(\Z/2m_\Lambda\)-grading on all the hom spaces, where \(m_\Lambda\) is the gcd of the minimal Maslov numbers of all connected components of \(\Lambda\).
\end{rem}

Next, we give the functorial property of \(\mathcal{O}_{\Lambda,f_i,J}\) under Legendrian isotopies of \(\Lambda\) and different choices of \(f_i\) and \(J\). We start with some algebraic constructions analogous to \cite{NRSSZ}*{Definition 3.19} and the construction that follows in \cite{NRSSZ}.

\begin{defn}\label{linkcompatible}
	Given two systems of semi-free DGAs \(\{\mathcal{A}^P\}_{P\subset\N}\), \(\{\mathcal{B}^P\}_{P\subset\N}\), we say that a set of DGA morphisms \[f^P:(\mathcal{A}^P,\partial)\to(\mathcal{B}^P,\partial)\] is \textit{compatible} if they satisfy the following three conditions:
	\begin{enumerate}
		\item The \(f^P\) are compatible with the link gradings: each \(f^P\) sends each generator \(a\) to a \(\Z\)-linear combination of composable words in \(\mathcal{A}^P\) from \(c(a)\) to \(r(a)\), i.e., words of the form \(x_1\cdots x_k\) with \(c(x_{i+1})=r(x_{i})\) for \(i=1,...,k-1\), and \(c(x_1)=c(a),r(x_k)=r(a)\).
		\item For any \(P\subset P'\), the following diagram commutes:
		\[\begin{tikzcd}[column sep=small,row sep=small]
			{\mathcal{A}^{P}} && {\mathcal{B}^{P}} \\
			\\
			{\mathcal{A}^{P'}_P} && {\mathcal{B}^{P'}_P.}
			\arrow["{f^{P}}", from=1-1, to=1-3]
			\arrow[from=1-1, to=3-1]
			\arrow[from=1-3, to=3-3]
			\arrow["{f^{P'}}", from=3-1, to=3-3]
		\end{tikzcd}\]
	\end{enumerate}
\end{defn}

\begin{defn}\label{auxiliary functor definition}
Given a set of morphisms \(f^P:(\mathcal{A}^P,\partial)\to(\mathcal{B}^P,\partial)\) that is compatible, we define a strictly unital \(A_\infty\)-functor \(F:\mathcal{O}_{\mathcal{B}}\to\mathcal{O}_{\mathcal{A}}\) as follows: on objects, \(F(i,\epsilon)=(i,\epsilon\circ f^{\{i\}})\). Next, we need to define maps \[F_k:\mathcal{O}_{\mathcal{B}}((i_{1},\epsilon_{1}),(i_{2},\epsilon_{2}))\otimes...\otimes\mathcal{O}_{\mathcal{B}}((i_{k},\epsilon_{k}),(i_{k+1},\epsilon_{k+1}))\to\mathcal{O}_{\mathcal{A}}((i_{1},\epsilon_{1}),(i_{k+1},\epsilon_{k+1})).\] Without loss of generality, assume that \(i_{1}<...<i_{k+1}\). In other cases, either the map is determined by strict unitality, or has trivial domain. Write \(P=\{i_1,...,i_{k+1}\}\). Consider the diagonal augmentation \(\epsilon=(\epsilon_{1},...,\epsilon_{k+1})\) of \(\mathcal{B}^{P}\), and let \(f^P_\epsilon:=\phi_\epsilon\circ f^P\circ \phi^{-1}_{\epsilon\circ f^P}\). Here, we used that \(f^P\) passes to a well-defined map \((\mathcal{A}^P)^{\epsilon\circ f^P}\to(\mathcal{B}^P)^{\epsilon}\), which can be shown by a direct computation. Another direct computation shows that \(f^P_{\epsilon}((\mathcal{A}^P)^{\epsilon\circ f^P}_{+})\subset (\mathcal{B}^P)^{\epsilon}_{+}\), i.e., no constant terms appear in the image of generators. We then define \(F_k\), up to the usual degree \(1\) grading shift (\(C^\vee=C^*[-1]\), from Section \ref{section augmented DGAs}), by dualizing the component of \(f^S_\epsilon\) that maps from \[C^{i_1 i_{k+1}}_{\mathcal{A}}\to C^{i_1 i_2}_{\mathcal{B}}\otimes...\otimes C^{i_k i_{k+1}}_{\mathcal{B}}.\]
\end{defn}

Note that \(F\) is an \(A_\infty\)-functor, since all the relations \(F_k\) is required to satisfy follow from the identity \(f^{P}_\epsilon\partial_{\epsilon\circ f^P}=\partial_\epsilon f_\epsilon^{f^P}\) for various finite sets \(P\) with \(|P|=k+1\).

Next we show the functorial property of \(\mathcal{O}_{\Lambda,f_i,J}\) under Legendrian isotopies of \(\Lambda\), and different choices of \(f_i\) and \(J\), which will be useful in the proof of Theorem \ref{invariance of aug+}. Suppose there is a Legendrian isotopy from \(\Lambda\) to \(\Lambda'\). Let  \(J,J'\) be almost complex structures compatible with the neighborhood \(U_\Lambda,U_{\Lambda'}\) respectively. Let \(f_1<f_2<\cdots:\Lambda\to\R\) and \(f_1'<f_2'<\cdots:\Lambda'\to\R\) be small functions. Let \(\Lambda\subset U_\Lambda\) and \(\Lambda'\subset U_\Lambda'\) be small neighborhoods, and write \(\Lambda^i=\Gamma_{df_i,\Lambda}\), \((\Lambda')^i=\Gamma_{df'_i,\Lambda'}\).

Then, there is a Legendrian isotopy from \(\coprod \Lambda_{i}\) to \(\coprod \Lambda'_{i}\), and a path of almost complex structures from \(J\) to \(J'\), admissible in the sense of \cite{EES1} and \cite{EES2}. Write \(\mathcal{A}^P\) (resp. \(\mathcal{B}^P\)) for the DGA of \(\coprod_{i\in P} \Lambda_{i}\) (resp. \(\coprod_{i\in P} \Lambda'_{i}\)). The proof of stable tame isomorphism invariance of LCH from \cite{EES1} and \cite{EES2} gives a set of stable tame isomorphisms \(f^P:\mathcal{A}^P\to\mathcal{B}^P\).

\begin{lem}\label{auxiliary functor lemma}
	The \(f^P:\mathcal{A}^P\to\mathcal{B}^P\) coming from Legendrian isotopies as above form a set of compatible DGA morphisms, i.e., they satisfy the properties in Definition \ref{linkcompatible}.
\end{lem}
\begin{proof}
	The proof of this lemma involves explicitly digging through the proof of stable tame isomorphism invariance of LCH in \cite{EES1}*{section 2}.

	Property (2) holds, since as algebra maps, \(f^{P}\) is simply the restriction of \(f^{P'}\) to the subalgebra \(\mathcal{A}^{P}\). To check property (1), it suffices to check it for Legendrian isotopies of \(\coprod_{i\in P} \Lambda_{i}\) consisting of a single handle slide, and Legendrian isotopies that create/annihilate a pair of Reeb chords.

	During a handle slide, a tame isomorphism happens: a single Reeb chord generator \(a\) gets mapped to \(a+A_0b_1A_1b_1...A_{k-1}b_kA_k\), where the term \(A_0b_1A_1b_1...A_{k-1}b_kA_k\) is determined by a certain holomorphic disk in \(\mathcal{M}_{\bar{A}}(a;b_1,...,b_k)\) in the auxiliary Legendrian submanifold in \cite{EES1}*{section 10.1}. Since \(A_0b_1A_1b_1...A_{k-1}b_kA_k\) comes from such a holomorphic disk, it is a composable word in \(\mathcal{A}^S\) from \(r(a)\) to \(c(a)\).

	During a Reeb chord creation, we need to look at the proof of \cite{EES1}*{Lemma 2.13}. We use terminologies and symbols in the proof of \cite{EES1}*{Lemma 2.13} freely here. The induced map of DGAs is a composition of a stabilisation and a tame isomorphism. A stabilisation clearly satisfies the property (1). The tame isomorphism is \(\Phi_l\), which is inductively defined, so we will prove that \(\Phi_l\) satisfies property (1) by induction. By its definition, \(\Phi_0(d)\) is clearly a composable word from \(c(d)\) to \(r(d)\), at least for generators \(d\neq b\). \(\Phi_0(b)=e_2^j+v\), which is given by a holomorphic disk on the auxiliary Legendrian, so it is a composable word from \(c(b)\) to \(r(b)\) by the same argument as the argument for handle slides. Assume \(\Phi_{i-1}\) satisfies property (2). Then, each term in \(\Phi_i(d)\) is either a Reeb chord from \(c(d)\) to \(r(d)\), or is \(H\circ \Phi_{i-1}\circ\partial(d)\), and we know that \(\partial\) is compatible with the link grading, and \(\Phi_{i-1}\) is compatible with the link grading by induction hypothesis, and \(H\) is also compatible with the link grading, since \(e_1^j\) and \(e_2^j\) have the same link gradings.	
\end{proof}

Combining Definition \ref{auxiliary functor definition} and Lemma \ref{auxiliary functor lemma}, we can define:

\begin{defn}\label{definition of functor F}
	Let \(\Lambda\) to \(\Lambda'\) be Legendrian isotopic. Let \(f_i:\Lambda\to\R\), \(f'_i:\Lambda'\to\R\). Let  \(J,J'\) be almost complex structures compatible with the neighborhoods \(U_\Lambda,U_{\Lambda'}\). Define \(F: \mathcal{O}_{\Lambda',f'_i,J'}\to \mathcal{O}_{\Lambda,f_i,J}\) as follows: since Lemma \ref{auxiliary functor lemma} guarantees us a set of compatible DGA morphisms, we can use Definition \ref{auxiliary functor definition} to define an \(A_\infty\)-functor \(F: \mathcal{O}_{\Lambda',f'_i,J'}\to \mathcal{O}_{\Lambda,f_i,J}\).
\end{defn}

The functor \(F: \mathcal{O}_{\Lambda',f'_i,J'}\to \mathcal{O}_{\Lambda,f_i,J}\) in Definition \ref{definition of functor F} is almost cohomologically fully faithful:

\begin{prop}\label{oinv}
	\(F_1:\mathcal{O}_{\Lambda',f'_i,J'}((i,\epsilon),(j,\epsilon'))\to \mathcal{O}_{\Lambda,f_i,J}((i,\epsilon),(j,\epsilon'))\) is a quasi isomorphism if \(i<j\).
\end{prop}
\begin{proof}
	\(F_1:\mathcal{O}_{\Lambda',f'_i,J'}((i,\epsilon),(j,\epsilon'))\to \mathcal{O}_{\Lambda,f_i,J}((i,\epsilon),(j,\epsilon'))\) is either the zero map between zero spaces (this is the case where \(i\geq j\)), or \(F_1\) is a component of the dual of the map \(f^{\{i,j\}}_\epsilon:C_{\mathcal{A}^{\{i,j\}}}\to C_{\mathcal{B}^{\{i,j\}}}\). In the latter case, since \(f^{\{i,j\}}_\epsilon\) is a quasi-isomorphism (since it is induced by a stable tame isomorphism of DGAs), \(F_1\), being a direct summand, is also a quasi-isomorphism.
\end{proof}

\subsection{The augmentation category}

Remember that in the construction before Definition \ref{defn of preaugmentation category}, we required our perturbations to be small, so all the \(\mathcal{A}^{\{i\}}\) are canonically isomorphic, so their augmentations are in canonical bijection. The pre-augmentation category does not have enough morphisms: since we set out to construct a category whose objects are augmentations of \(\Lambda\), all objects corresponding to a single augmentation morally should be in the same isomorphism class. However, in the pre-augmentation category, there is no morphism from \((2,\epsilon)\) to \((1,\epsilon)\). To fix this, and to obtain the actual augmentation category, we localise. For each choice of \(i\), the Reeb chords of \(\Lambda^{i}\cup\Lambda^{i+1}\) contain the set of short Reeb chords, corresponding to the critical points of a Morse function \(f_{i+1}-f_{i}:\Lambda\to \R\). Define \[w_{i,\epsilon}^{\vee}=\sum_l p_l^{\vee}\in \mathcal{O}_{\Lambda,f_i,J}((i,\epsilon),(i+1,\epsilon)),\] where \(\{p_l\}\) is the set of all minimums of \(f_{i+1}-f_{i}\). We would like to localise \(\mathcal{O}_{\Lambda,f_i,J}\) at \(w_{i,\epsilon}^{\vee}\) for all \(i\geq 1\), and to do that, we need to prove that each \(w_{i,\epsilon}\) is a \(0\)-cycle.

For technical details about localisation of an \(A_\infty\)-category (at a set of degree 0 cohomology classes of morphisms), see \cite{GPS1}*{section 3.1.3}. 

\begin{lem}\label{cocycle}
	Let \(\epsilon:\mathcal{A}^{\{1\}}\to k\) be an augmentation. Then, \[w_{i,\epsilon}^{\vee}\in\mathcal{O}_{\Lambda,f_i,J}((i,\epsilon),(i+1,\epsilon))\] is a \(0\)-cocycle.
\end{lem}
\begin{proof}
	We have that \(m_1(\epsilon):C_{i,i+1}^{\vee}\to C_{i,i+1}^{\vee}\) is given by \[m_1(\epsilon)(p_l^{\vee})=\sum_{a\in\mathcal{R}^{i,i+1}} a^{\vee}\cdot \Coeff_{p_l}(\partial_\epsilon a),\] where \(\mathcal{R}^{i,i+1}\) is the set of all Reeb chords that starts on \(\Lambda_{i+1}\) and end on \(\Lambda_i\).

	To prove that \(w_{i,\epsilon}^{\vee}\) is a cocycle, we use quantum flow trees. For any index 1 critical point \(q\) of \(f_{i+1}-f_{i}\), let \(\gamma_q\) denote the loop on \(\Pi(\Lambda)\) formed by the capping path from the basepoint to \(q\), the flow line from \(q\) to an index \(0\) critical point \(p\), and the inverse of the capping path from the basepoint to \(p\). By Lemma \ref{disks to a minimum in 2 sheets}, the rigid holomorphic disks that have a negative puncture at some \(p_l\) are rigid complete flow lines of \(f_{i+1}-f_{i}\) that start from some index 1 critical point and end at \(p_l\), and multiscale flow trees consisting of a constant disk with a partial flow line that flows to \(p_l\).
	
	For rigid complete flow lines, note that each such flow line contributes \(\epsilon(\gamma_q)\epsilon(\gamma_q^{-1})=1\). The total contribution is 0, since the sum of all minimums of a Morse function is a cocycle in the Morse cochain complex.
	
	The contribution from the latter rigid quantum flow trees cancel out in pairs: for each Reeb chord \(b\) of \(\Lambda\), there are four corresponding long Reeb chords in \(\Lambda^{\{i,i+1\}}\), denoted \(b_{i,i},b_{i,i+1},b_{i+1,i},b_{i+1,i+1}\) (the first subscript is the number of the beginning sheet, and the second subscript is the number of the ending sheet). Since we are only interested in the differential in \(C_{i,i+1}^{\vee}\), we only look at rigid holomorphic disks with a positive puncture at \(b_{i,i+1}\). Given a quantum flow tree consisting of a constant disk at \(b\) and a Morse flow line that flows to some minimum, there are two lifts to a disk. One consists of a negative puncture at a minimum \(p\) and a negative puncture at \(b_{i,i}\); the other one consists of a negative puncture at a minimum and a negative puncture at \(b_{i+1,i+1}\). Let \(\gamma_b\) denote the loop on \(\Pi(\Lambda)\) formed by the capping path from the basepoint to \(p\), the inverse of the partial flow line from the positive endpoint of \(b\) to \(p\), and the inverse of the capping path from the basepoint to the positive endpoint of \(b\). Also, let \(\gamma_b'\) denote the loop on \(\Lambda\) formed by the capping path from the basepoint to the negative endpoint of \(b\), the partial flow line from the negative endpoint of \(b\) to \(p\), and the inverse of the capping path from the basepoint to \(p\). When computing \(m_1(\epsilon)(w_{i,\epsilon}^{\vee})\), one of these rigid quantum flow trees contribute the summand \(\epsilon(b)\epsilon(\gamma_b)\epsilon(\gamma_b^{-1})b_{i,i+1}^{\vee}=\epsilon(b)b_{i,i+1}^{\vee}\), and the other rigid quantum flow tree contributes the summand \(\epsilon(b)\epsilon(\gamma_b')\epsilon({\gamma_b'}^{-1})b_{i,i+1}^{\vee}=\epsilon(b)b_{i,i+1}^{\vee}\), so they cancel out and contribute 0.
	
	So, \(w_{i,\epsilon}^{\vee}\) is a cocycle.

	We also need to prove that the index \(0\) critical points \(p_l\) have degree \(0\). The construction before Definition \ref{defn of preaugmentation category} specifies choices of paths connecting short Reeb chord endpoints to the basepoints and the paths that connect different basepoints. We compute the grading by assembling these paths into a loop in \(M\), and look at the corresponding loop in the Lagrangian Grassmanian. The loop in \(M\) we get is the concatenation of: a path \(\gamma\) from \(\Pi(p_l)\) to \(\Pi(b_j^i)\) (where \(b_j^i\) is the basepoint of the component of \(p_l\) on \(\Lambda^i\)), a path from \(\Pi(b_j^i)\) to \(\Pi(b_1^i)\) that is \(C^0\)-close to \(\gamma_{1j}\), a path from \(\Pi(b_1^i)\) to \(\Pi(b_1^1)\) that is \(C^0\) small, the inverse of the path from \(\Pi(b_1^{i+1})\) to \(\Pi(b_1^1)\) that is \(C^0\) small, a path from \(\Pi(b_1^{i+1})\) to \(\Pi(b_j^{i+1})\) that is \(C^0\)-close to the inverse of \(\gamma_{1j}\), and the inverse of the path from \(\Pi(p_l)\) to \(\Pi(b_j^{i+1})\) (which is \(C^0\)-close to the inverse of the path from \(\Pi(p_l)\) to \(\Pi(b_j^i)\)). Whenever two paths are \(C^0\)-close above, we have also arranged in the construction so that the corresponding paths of Lagrangian subspaces are also \(C^0\)-close. We see that the loop in \(M\), and the corresponding loop in the Lagrangian Grassmanian obtained in the end are contractible loops, therefore the Maslov class is 0. The degree of \(p_l\) is therefore \(0-1=-1\) in the Chekanov-Eliashberg DGA. Since there is a grading shift by 1 degree going from the DGA to \(C^\vee\) (see Section \ref{section augmented DGAs}), the degree of \(p_l\) in \(\mathcal{O}_{\Lambda,f_i,J}((i,\epsilon),(i+1,\epsilon))\) is \(0\).
\end{proof}

\begin{defn}\label{aug}
	The \textit{augmentation category} \(\mathcal{A}ug_+(V,\Lambda)\) is defined to be the localisation \(\mathcal{O}_{\Lambda,f_i,J}[W^{-1}]\), where \(W:=\{[w_{i,\epsilon}^{\vee}]\in H^0 \mathcal{O}_{\Lambda,f_i,J}((i,\epsilon),(i+1,\epsilon))\mid i\geq 1\}\).
\end{defn}

Note that \(\mathcal{A}ug_+(V,\Lambda)\) is cohomologically unital, since the localisation of a cohomologically unital \(A_\infty\)-category is cohomologically unital. Now we prove the invariance of \(\mathcal{A}ug_+\).

\begin{prop}\label{locinvlem}
	Let \(F:\mathcal{O}_{\Lambda',f'_i,J'}\to\mathcal{O}_{\Lambda,f_i,J}\) be the functor defined in Definition \ref{definition of functor F}. Then, \[H^*F([(w_{i,\epsilon}^{\vee})'])=[w_{i,F(\epsilon)}^{\vee}].\]
\end{prop}

To prove Proposition \ref{locinvlem}, we put a multiplication map on \(H^*\mathcal{O}_{\Lambda,f_i,J}((i,\epsilon),(i,\epsilon+1))\), and prove that \([w_{i,F(\epsilon)}^{\vee}]\) is the unit element. Without loss of generality, assume \(i=1\).

Let \(f,h:\Lambda\to\R\) be small positive functions. Choose \[f_1=0, f_2=f, f_3=f+\epsilon f+\epsilon^2 h,\] and \[f'_1=0, f'_2=\epsilon f, f'_3=f+\epsilon f+\epsilon^2 h.\] Perturb \(f,h\) so that \(\coprod_{i=1,2,3} \Lambda^{i}\) and \(\coprod_{i=1,2,3} \Lambda'^{i}\) are both quantum flow tree-generic. If \(\epsilon\) is small enough, then throughout the straight line homotopy \(f^t_i\) from \(f_i\) to \(f'_i\) (where only \(f_2\) is changing), we have that the functions
\begin{align*}
f^t_3-f^t_2&=(1+\epsilon-t)f+\epsilon^2 h,\\
f^t_3-f^t_1&=f+\epsilon f+\epsilon^2 h,\\
f^t_2-f^t_1&=tf
\end{align*}
are up to scaling, all small perturbations of the function \(f\). Throughout the homotopy, for \(1\leq i,j\leq 3\), we have that \[\mathcal{O}_{\Lambda,f^t_i,J}((i,\epsilon),(j,\epsilon))\] are canonically isomorphic, for all \(t\in[1,\epsilon]\) (this is because when \(\Lambda\) stays the same during an isotopy, Reeb chords and quantum flow trees only depend on the function between the 2 sheets, up to small perturbation).

Taking cohomology, \(m_2\) gives a map \[H^* m_2:H^*\mathcal{O}_{\Lambda,f_i,J}((1,\epsilon),(2,\epsilon))\otimes H^*\mathcal{O}_{\Lambda,f_i,J}((2,\epsilon),(3,\epsilon))\to H^*\mathcal{O}_{\Lambda,f_i,J}((1,\epsilon),(3,\epsilon)).\] Applying the stable tame isomorphism invariance of the Chekanov-Eliashberg DGA from \cite{EES2} to the DGA of the 3-copy Legendrian \(\Lambda^1\cup\Lambda^2\cup\Lambda^3\), we have that \(m_2\) is natural with respect to Legendrian isotopies, up to quasi-isomorphism. So, \(H^*m_2\) is natural with respect to Legendrian isotopies. This gives rise to a multiplication map \[M:H^*\mathcal{O}_{\Lambda,f_i,J}((1,\epsilon),(2,\epsilon))\otimes H^*\mathcal{O}_{\Lambda,f_i,J}((1,\epsilon),(2,\epsilon))\to H^*\mathcal{O}_{\Lambda,f_i,J}((1,\epsilon),(2,\epsilon)).\] A priori, if we replace \(f_i\) by \(f'_i\), we might obtain different multiplication maps. However, as we have shown above, the straight line homotopy between the 2 choices induces canonical isomorphisms of groups between the 2 choices. We then use the fact that \(H^* m_2\) is natural with respect to different choices of \(f_i\) (which can be realised by Legendrian isotopies) to conclude that the 2 choices give the same multiplication map \(M\). Note that the multiplication map only depends on \(f_2-f_1\), and it does not depend on \(f_i\) for any other \(i\).

Since \(H^* m_2\) is natural with respect to Legendrian isotopies and different choices of \(f_i\), \(M\) is also natural with respect to Legendrian isotopies and different choices of \(f_i\).

\begin{prop}\label{unit lemma}
	\([w_{1,\epsilon}^{\vee}]\in H^0 \mathcal{O}((1,\epsilon),(2,\epsilon))\) is the unit for \(M\).
\end{prop}
\begin{proof}
	To show that \([w_{1,\epsilon}^{\vee}]\) is a right unit for \(M\), we use \(f_1,f_2,f_3\) as above. We have that \(\mathcal{O}_{\Lambda,f_i,J}((1,\epsilon),(3,\epsilon))\), \(\mathcal{O}_{\Lambda,f_i,J}((2,\epsilon),(3,\epsilon))\) are canonically isomorphic, since \(\Lambda^{1,3}\) is a perturbation of \(\Lambda^{2,3}\).

	Now we compute the product \(M(\cdot,[w_{1,\epsilon}^\vee])\) using quantum flow trees. For quantum flow trees here, we use Corollary \ref{disks to a minimum in 3 sheets}, and we degenerate \(\Lambda^{\{1,2,3\}}\) onto \(\Lambda^{\{2,3\}}\) (see the discussion after Corollary \ref{disks to a minimum in 3 sheets}), since \(\Lambda^1\) and \(\Lambda^2\) are close together.

	We have \((w_{1,\epsilon}^{\vee})=\sum (p_l)_{12}^\vee\), where \((p_l)_{12}\) are minimums of \(f_2-f_{1}\). By Corollary \ref{disks to a minimum in 3 sheets}, for rigid quantum flow trees with a negative puncture at some \((p_l)_{12}\), there are two possibilities: a rigid flow line that starts from an index 1 critical point of \(f_{12}\) and ends at \((p_l)_{12}\), or a constant holomorphic disk with a small flow line flowing down to \((p_l)_{12}\).

	For a Reeb chord \(b_{23}\), when computing \(m_2(w_{1,\epsilon}^{\vee},b_{23}^{\vee})\), we count rigid quantum flow trees with a negative puncture at some \((p_l)_{12}\), and a negative puncture at \(b_{23}\). Since we are composing a minimum \(p_l\) on the right of \(b_{23}^{\vee}\), we only consider quantum flow trees with a partial flow line to the right of \(b_{23}^{\vee}\). Since the positive endpoint of \(b\) flows to a unique minimum point \(p_l\), there is precisely one such rigid quantum flow tree, which is a constant disk at \(b\), with a small flow line flowing down to a particular \(p_l\). After lifting the quantum flow tree, we get a disk with positive puncture at \(b_{13}\), and negative punctures at \((p_l)_{23}\) and \(b_{12}\). So we have that \(m_2(w_{1,\epsilon}^{\vee},b_{23}^{\vee})=b_{13}\) (similarly to the computation of \(m_1(w_{1,\epsilon}^{\vee})\) in the proof of Lemma \ref{cocycle}, the contributions from the homology classes cancel out). So, \((w_{1,\epsilon}^{\vee})\) is a right unit.

	A similar argument using \(f'_1,f'_2,f'_3\) shows that \([w_{1,\epsilon}^{\vee}]\) is a left unit.
\end{proof}

\begin{proof}[Proof of Proposition \ref{locinvlem}]
	Applying the stable tame isomorphism invariance of the Chekanov-Eliashberg DGA from \cite{EES2} to the DGA of the 3-copy Legendrian \(\Lambda^1\cup\Lambda^2\cup\Lambda^3\), we have that \(m_2\) is natural with respect to Legendrian isotopies, up to quasi-isomorphism. This implies that the multiplication \(M\) is natural with respect to Legendrian isotopies. In other words, \[H^*(F)\circ M=M\circ (H^*(F)\otimes H^*(F)).\] So, for \[a,b\in H^*\mathcal{O}_{\Lambda',f'_i,J'}((1,\epsilon),(2,\epsilon)),\] we have that \[M(H^*(F)(a)\otimes H^*(F)(b))=H^*(F)(M(a\otimes b)).\] Furthermore, by Proposition \ref{oinv}, \[H^*(F):H^*\mathcal{O}_{\Lambda',f'_i,J'}((1,\epsilon),(2,\epsilon))\to H^*\mathcal{O}_{\Lambda,f_i,J}((1,F(\epsilon)),(2,F(\epsilon)))\] is a bijection. The lemma then follows, since a bijection that preserves a multiplication map sends the identity to the identity. 
\end{proof}

\begin{lem}\label{new lemma}
	For any \(i,j\), \(H^*\mathcal{A}ug_+((i,\epsilon),(j,\epsilon'))\cong H^*\mathcal{O}_{\Lambda,f_i,J}((1,\epsilon),(2,\epsilon))\).
\end{lem}

\begin{proof}
	By an argument similar to the proof of Proposition \ref{unit lemma}, for \(i<j\), the map \[H^* \mathcal{O}((i,\epsilon),(j,\epsilon'))\to H^* \mathcal{O}((i,\epsilon),(j+1,\epsilon'))\] given by composition by \(w_{j,\epsilon'}^\vee\) is an isomorphism. We then apply \cite{GPS1}*{Lemma 3.16} to conclude the proof.
\end{proof}

Localisation of \(A_\infty\)-categories satisfy the following universal property (see \cite{Sei}*{(2.40)}): Let $\mathcal{B}$ be an $A_{\infty}$-category, and $fun(\mathcal{A}, \mathcal{B})$ the $A_{\infty}$-category of $A_{\infty}$-functors. Composition with the localisation functor $\mathcal{A} \rightarrow S^{-1} \mathcal{A}$ induces a cohomologically full and faithful $A_{\infty}$-functor $fun\left(S^{-1} \mathcal{A}, \mathcal{B}\right) \rightarrow fun(\mathcal{A}, \mathcal{B})$, whose essential image are the functors that map elements of $S$ to isomorphisms on the cohomology level.

\begin{thm}\label{invariance of aug+}
	Let \(\Lambda\) and \(\Lambda'\) be Legendrian isotopic, and let \(f_1<f_2<\cdots:\Lambda\to\R\) and \(f'_1<f'_2<\cdots :\Lambda'\to\R\) be generically chosen small functions, and let \(J,J'\) be almost complex structures compatible with the neighborhoods \(U_\Lambda,U_{\Lambda'}\). Then, \(\mathcal{O}_{\Lambda,f_i,J}[W^{-1}]\) is \(A_\infty\)-equivalent to \(\mathcal{O}_{\Lambda',f'_i,J'}[W'^{-1}]\). That is,
	\(\mathcal{A}ug_+(V,\Lambda)\) is well-defined, and is invariant under Legendrian isotopy.
\end{thm}
\begin{proof}
	Assume that \(\Lambda\) is isotopic to \(\Lambda'\) via either a handle slide, a Reeb chord birth, or a Reeb chord death. By Definition \ref{definition of functor F}, we have an \(A_\infty\)-functor \(F:\mathcal{O}_{\Lambda',f'_i,J'}\to\mathcal{O}_{\Lambda,f_i,J}\). Write \(W=\{[w_{i,\epsilon}^{\vee}]\in H^0 \mathcal{O}_{\Lambda,f_i,J}((i,\epsilon),(i+1,\epsilon))|i\geq1\}\), and \(W'=\{[(w_{i,\epsilon}^{\vee})']\in H^0 \mathcal{O}_{\Lambda',f'_i,J'}((i,\epsilon),(j,\epsilon))|i\geq1\}\). By Proposition \ref{locinvlem}, \(H^*F(W')\subset W\). So, the composition of functors
	\[\begin{tikzcd}[column sep=small]
	{\mathcal{O}_{\Lambda',f'_i,J'}} && {\mathcal{O}_{\Lambda,f_i,J}} && {\mathcal{O}_{\Lambda,f_i,J}[W^{-1}]}
	\arrow["F", from=1-1, to=1-3]
	\arrow["\imath", from=1-3, to=1-5]
	\end{tikzcd}\]
	sends \(W'\) to cohomologically invertible morphisms, which by the universal property of localisation, induces a functor \[\bar{F}: \mathcal{O}_{\Lambda',f'_i,J'}[(W')^{-1}]\to \mathcal{O}_{\Lambda,f_i,J}[W^{-1}].\]
	
	Lemma \ref{new lemma} together with Proposition \ref{oinv} show that \(H^*\bar{F}\) is fully faithful. If the isotopy from \(\Lambda\) to \(\Lambda'\) is either a handle slide or a Reeb chord birth, then \(\bar{F}\) is surjective on objects. So, \(\bar{F}\) is an \(A_\infty\)-equivalence if the isotopy from \(\Lambda\) to \(\Lambda'\) is either a handle slide or a Reeb chord birth. The case where the isotopy from \(\Lambda\) to \(\Lambda'\) is a Reeb chord death is automatic, since \(A_\infty\)-equivalence is an equivalence relation, and the isotopy from \(\Lambda'\) to \(\Lambda\) is a Reeb chord birth.
\end{proof}

\section{Equivalence with the consistent sequence construction}\label{equivalence with consistent sequences}

Next, we relate our construction of \(\mathcal{A}ug_+\) using localisation to the construction of \(\mathcal{A}ug_+\) using a consistent sequence of DGAs given in \cite{NRSSZ}*{section 3.1-3.3, section 4.1-4.2}, by proving that they are \(A_\infty\)-equivalent. We also give a similar alternative characterisation of the cohomology category \(H^*\mathcal{A}ug_+\).

Let \(\mathcal{C}\) be a cohomologically unital \(A_\infty\)-category over a field \(k\). Construct a cohomologically unital \(A_\infty\)-category \(\mathcal{C}'\) where each object is of the form \[(i,X)\in\N_{\geq1}\times Ob(\mathcal{C}),\] and with morphism spaces \[\mathcal{C}'((i,X),(j,Y)):=(i,j,\mathcal{C}(X,Y)),\] and the multiplication maps are simply the multiplication maps in \(\mathcal{C}\), ignoring the \(\N_{\geq1}\) factor.

Construct a strictly unital \(A_\infty\)-category \(\mathcal{C}''\), with the same objects as \(C'\), and for \(i<j\), \[\mathcal{C}''((i,X),(j,Y)):=\mathcal{C}'((i,X),(j,Y)),\] and endomorphisms \[\mathcal{C}'((i,X),(i,X))=k,\] and otherwise \[\mathcal{C}'((i,X),(j,Y))=0,\] with the obvious compositions that make \(\mathcal{C}''\) strictly unital.

\begin{prop}\label{triviallemma}
	Let \(A:=\{(i,i+1,id_X)\in \mathcal{C}''((i,X),(i+1,X))|i\geq1,X\in\mathcal{C}\}\). Then, \(\mathcal{C}''[A^{-1}]\) is \(A_\infty\)-equivalent to \(\mathcal{C}\).
\end{prop}
\begin{proof}
	\(\mathcal{C}\) is equivalent to \(\mathcal{C}'\), by sending \(X\mapsto(1,X)\), and higher maps given by the higher multiplication maps in \(\mathcal{C}\).

	There is an obvious inclusion functor \(F:\mathcal{C}''\to\mathcal{C}'\). This functor sends morphisms in \(A\) to isomorphisms, so by the universal property of localisation, it is cohomologically naturally isomorphic to the composition 
	\[\begin{tikzcd}[column sep=small]
		{\mathcal{C}''} && {\mathcal{C}''[A^{-1}]} & {} & {\mathcal{C}'}
		\arrow["{\imath}",from=1-1, to=1-3]
		\arrow["{\bar{F}}", from=1-3, to=1-5]
	\end{tikzcd}\]

	For every \((j,X)\), pre-composing with \((i,i+1,id_X)\) and taking cohomology induces a map \[H^*C''((j+1,X),(i,Y))\to H^*C''((j,X),(i,Y))\] which is an isomorphism for \(j>i+1\), since \(id_X\) is a cohomological unit. So, taking direct limits, we have that \[\lim_{\longrightarrow i}H^*C''((j+1,X),(i,Y))\to \lim_{\longrightarrow i}H^*C''((j,X),(i,Y))\] is an isomorphism. By \cite{GPS1}*{Lemma 3.16}, we have that \[\imath:\lim_{\longrightarrow i}H^*C''((l,Z),(i,Y))\to\lim_{\longrightarrow i}H^*(C''[A^{-1}])((l,Z),(i,Y))\] is an isomorphism for all \((l,Z)\). Both direct limits are very simple: for the right-hand side, we have that \(H^*(C''[A^{-1}])((l,Z),(i,Y))\) is a sequence of isomorphisms (since it is obtained by pre-composing with isomorphisms); for the left-hand side, we have that for \(i>l\), \(H^*C''((l,Z),(i,Y))\) is a sequence of isomorphisms. So, we have that for \(i>l\), \[\imath:H^*C''((l,Z),(i,Y))\to H^*(C''[A^{-1}])((l,Z),(i,Y))\] is an isomorphism. We also know that for \(i>l\), \[F=\bar{F}\circ\imath:H^*C''((l,Z),(i,Y))\to H^* C'((l,Z),(i,Y))\] is an isomorphism. So, for \(i>l\), \[\bar{F}:H^*(C''[A^{-1}])((l,Z),(i,Y))\to H^* C'((l,Z),(i,Y))\] is an isomorphism. But in both \(C''[A^{-1}]\) and \(C'\), \((i,Y)\) is (cohomologically) isomorphic to \((j,Y)\) for all \(i,j\). Conjugating by appropriate isomorphisms, we see that for any \((l,Z),(j,Y)\), \[\bar{F}:H^*(C''[A^{-1}])((l,Z),(i,Y))\to H^* C'((l,Z),(i,Y))\] is an isomorphism. Therefore, \(\bar{F}\) is cohomologically fully faithful, and is clearly essentially surjective. 
\end{proof}

\begin{defn}\label{consistency}
	A system of semi-free DGAs \(\{\mathcal{A}^P\}_{P\subset\N}\) is said to be \emph{consistent} if for every order preserving bijection \(g:P\to P'\), there exists a bijection \(\mathcal{S}^P\to \mathcal{S}^{P'}\), such that the induced algebra homomorphism \(\mathcal{A}^P\to \mathcal{A}^{P'}\) is a DGA isomorphism. The maps \(\mathcal{S}^P\to \mathcal{S}^{P'}\) are required to be functorial (i.e. the assignment \(P\mapsto \mathcal{S}^P\) is a functor from the category of finite subsets of \(\N\) with order preserving \textit{bijections}, to the category of sets), and to be compatible with the link grading maps (i.e. the following diagram commutes):
	\[\begin{tikzcd}[column sep=small,row sep=tiny]
	{\mathcal{S}^P} && P \\
	\\
	{\mathcal{S}^{P'}} && {P'.}
	\arrow["{r,c}", from=1-1, to=1-3]
	\arrow[from=1-1, to=3-1]
	\arrow[from=1-3, to=3-3]
	\arrow["{r,c}", from=3-1, to=3-3]
	\end{tikzcd}\]
\end{defn}

\begin{rem}
	Given a consistent system of semi-free DGAs \(\{\mathcal{A}^P\}_{P\subset\N}\), we have that the sequence \[\mathcal{A}^{\{1\}},\mathcal{A}^{\{1,2\}},\mathcal{A}^{\{1,2,3\}},\cdots\] is a consistent sequence of DGAs, in the sense of \cite{NRSSZ}*{Definition 3.13}.
\end{rem}

Now, given a consistent system of semi-free DGAs \(\{\mathcal{A}^P\}_{P\subset\N}\) that comes from parallel copies of a Legendrian, we have two constructions of \(\mathcal{A}ug_+\): one from Definition \ref{aug}, and one from \cite{NRSSZ}*{Definition 3.16}.

\begin{thm}\label{equivalence of the two constructions}
	Given a consistent system of semi-free DGAs \(\{\mathcal{A}^P\}_{P\subset\N}\) that comes from parallel copies of a Legendrian, the two constructions of \(\mathcal{A}ug_+\) are \(A_\infty\)-equivalent.
\end{thm}

\begin{proof}
	In the language of Proposition \ref{triviallemma}:
	\begin{enumerate}
		\item the construction of \(\mathcal{A}ug_+\) from \cite{NRSSZ}*{Definition 3.16} corresponds to the category \(\mathcal{C}\);
		\item \(\mathcal{O}_{\Lambda,f_i,J}\) corresponds to \(\mathcal{C}''\);
		\item \(A\) corresponds to \(W\) (by a combination of Lemma \ref{cocycle} with Proposition \ref{unit lemma});
		\item therefore, the construction of \(\mathcal{A}ug_+\) from Definition \ref{aug} corresponds to the category \(\mathcal{C}''[A^{-1}]\).
	\end{enumerate}
	We then conclude by Proposition \ref{triviallemma} that the two constructions are \(A_\infty\)-equivalent.
\end{proof}

Therefore, if one obtains consistent systems, one could use those to define \(\mathcal{A}ug_+\), and a combination of Theorem \ref{equivalence of the two constructions} and Theorem \ref{invariance of aug+} shows that it is well-defined, and is invariant under Legendrian isotopy.

\begin{rem}\label{invariance of aug+ from NRSSZ}
	In the case of \(V=\R^3\), \cite{NRSSZ}*{Proposition 4.2} provides consistent systems. Furthermore, \cite{NRSSZ}*{Proposition 3.28} gives that minus the sum of all minimums is a strict unit, therefore implying the analogues of Lemma \ref{cocycle} and Proposition \ref{locinvlem} with signs taken into account. With this, we can regard section 3 as an alternative proof of invariance of \(\mathcal{A}ug_+(\R^3,\Lambda)\) constructed in \cite{NRSSZ} over any field.
\end{rem}

Sometimes in practice, one only cares about the cohomology category \(H^*\mathcal{A}ug_+(V,\Lambda)\), and not the full \(A_\infty\) structure. For example, isomorphism of objects in an \(A_\infty\)-category usually means ``quasi-isomorphism'', which only depends on the cohomology category. Below, we give an alternative characterisation of \(H^*\mathcal{A}ug_+(V,\Lambda)\).

The construction described after the statement of Proposition \ref{locinvlem} is a perturbation scheme of the 3-copy Legendrian \(\Lambda^1\cup\Lambda^2\cup\Lambda^3\), that is consistent, in the sense that the union of any two copies of \(\Lambda\) from the 3-copy have canonically isomorphic Chekanov-Eliashberg DGAs.

One can construct a category \(\mathcal{D}\) (an ordinary category, not an \(A_\infty\)-category) in the following way: objects of the category are augmentations of \(\Lambda\). Morphisms are the homology of the cochain complexes generated by Reeb chords from \(\Lambda^1\) to \(\Lambda^2\) (alternatively and isomorphically, Reeb chords from \(\Lambda^2\) to \(\Lambda^3\) or Reeb chords from \(\Lambda^1\) to \(\Lambda^3\)), equipped with the usual differential \(m_1\). Composition is given by \(H^*m_2\). This composition is a slight generalisation of the multiplication map \(M\) constructed after the statement of Proposition \ref{locinvlem}.

\begin{prop}
	The category \(\mathcal{D}\), described above, is equivalent to the cohomology category of the augmentation category \(H^*\mathcal{A}ug_+(V,\Lambda)\).
\end{prop}
\begin{proof}[Proof sketch]
	Since the author does not know how to construct a perturbation scheme that gives rise to a consistent sequence of DGAs in general, we use the following perturbation scheme that is consistent on the \(m_1,m_2\) level, which suffices for constructing the cohomology category.
	
	By picking all \(f_i\) to be scalings of the same function \(f\) (the function used in the construction after the statement of Proposition \ref{locinvlem}), and perturb generically, the system of DGAs we obtain satisfies a notion of consistency slightly weaker than the notion of consistency in Definition \ref{consistency}: we replace the condition of ``for every order preserving bijection \(g:P\to P'\)'' to ``for every order preserving bijection \(g:P\to P'\) for \(P,P'\) of cardinality \(1\) or \(2\)''.
	
	We mimic this entire section, up to the proof of Theorem \ref{equivalence of the two constructions}, by replacing \(\mathcal{C},\mathcal{C}',\mathcal{C}''\) with ordinary categories. Proposition \ref{triviallemma} still holds, by replacing ``\(A_\infty\)-equivalent'' with ``equivalent''. Then:
	\begin{enumerate}
		\item the category \(\mathcal{D}\), constructed above, corresponds to the category \(\mathcal{C}\);
		\item the category \(H^*\mathcal{A}ug_+\) corresponds to the category \(\mathcal{C}''[A^{-1}]\).
	\end{enumerate}
	The proof then follows from the ordinary category analogue of Proposition \ref{triviallemma}.
\end{proof}


\end{document}